\documentclass[a4paper,11pt]{amsart}
\usepackage{amssymb}
\usepackage{amsmath}
\usepackage{amsthm}
\usepackage{verbatim}
\usepackage[all]{xy}
\usepackage{url}

\usepackage{color}
\theoremstyle{plain}

\newtheorem{theorem}[equation]{Theorem}
\newtheorem{proposition}[equation]{Proposition}

\newtheorem{lemma}[equation]{Lemma}
\newtheorem{corollary}[equation]{Corollary}

\theoremstyle{definition}

\newtheorem{remark}[equation]{Remark}

\newtheorem{definition}[equation]{Definition}

\newtheorem{notation}[equation]{Notation}
\newtheorem{construction}[equation]{Construction}
\newtheorem{subsec}[equation]{}

\numberwithin{equation}{section}

\newcommand{\Z}{{\mathbb Z}}
\newcommand{\R}{{\mathbb R}}
\newcommand{\C}{{\mathbb C}}
\newcommand{\A}{{\mathbb A}}
\newcommand{\G}{{\mathbb G}}
\newcommand{\Q}{{\mathbb Q}}

\newcommand{\GG}{{\bf G}}
\newcommand{\YY}{{\bf Y}}
\newcommand{\HH}{{\bf H}}

\newcommand{\CC}{{\bf C}}
\newcommand{\TT}{{\bf T}}
\newcommand{\XX}{{\bf X}}

\newcommand{\X}{{\sf X}}

\newcommand{\gG}{{\scriptscriptstyle \bf G}}
\newcommand{\yY}{{\scriptscriptstyle \bf Y}}
\newcommand{\xX}{{\scriptscriptstyle \bf X}}

\newcommand{\Cc}{{\scriptscriptstyle C}}
\newcommand{\Hh}{{\scriptscriptstyle H}}

\newcommand{\into}{\hookrightarrow}

\DeclareMathOperator{\Inn}{Inn}
\DeclareMathOperator{\Aut}{Aut}
\DeclareMathOperator{\SAut}{SAut}
\DeclareMathOperator{\Out}{Out}
\DeclareMathOperator{\SOut}{SOut}

\DeclareMathOperator{\Lie}{Lie}

\newcommand{\Hom}{{\rm Hom}}

\newcommand{\Gal}{{\rm Gal}}
\newcommand{\ad}{{\rm ad}}

\newcommand{\inn}{{\rm inn}}

\newcommand{\ov}{\overline}

\newcommand{\sss}{{\rm ss}}
\newcommand{\ssc}{{\rm sc}}

\newcommand{\isoto}{\overset{\sim}{\to}}
\newcommand{\onto}{\twoheadrightarrow}
\newcommand{\labelto}[1]{\xrightarrow{\makebox[1.5em]{\scriptsize ${#1}$}}}

\newcommand{\hs}{\kern 0.8pt}
\newcommand{\hssh}{\kern 1.2pt}
\newcommand{\hshs}{\kern 1.6pt}
\newcommand{\hssss}{\kern 2.0pt}

\newcommand{\hm}{\kern -0.8pt}
\newcommand{\hmm}{\kern -1.2pt}

\newcommand{\emm}{\bfseries}

\renewcommand{\hbar}{{\bar h}}

\newcommand{\gam}{{\gamma}}
\newcommand{\Gam}{{\Gamma}}

\newcommand{\upgam}{\hs^\gamma\hm}
\newcommand{\upgams}{\hs^{\gamma*}}

\newcommand{\id}{{\rm id}}
\newcommand{\ups}{{\hs^s\hm}}

\newcommand{\spl}{{\rm spl}}

\newcommand{\fxo}{f_{x_0}}

\newcommand{\ff}{{\rm f}}
\newcommand{\uu}{{\rm u}}
\newcommand{\red}{{\rm red}}
\newcommand{\mmu}{{\rm mod-u}}

\newcommand{\hxx}{h_x}
\newcommand{\fx}{f_x}
\newcommand{\hgg}{h_{\gam,\gam}}
\newcommand{\fg}{f_\gam}
\newcommand{\yo}{{y_1}}
\newcommand{\gz}{g_{01}}

\newcommand{\fxu}{\fx}
\newcommand{\hxu}{\hxx}
\newcommand{\gyo}{g_{y_0}}

\newcommand{\im}{{\rm im\,}}
\newcommand{\cC}{{\scriptscriptstyle \CC}}

\newcommand{\hh}{{\mathfrak h}}

\newcommand{\Stab}{{\rm Stab}}

\newcommand{\SL}{{\rm SL}}
\newcommand{\Oo}{{\mathcal O}}

\newcommand{\GmC}{\G_{m,\C}}

\newcommand{\cc}{\raise 1.7pt \hbox{\Tiny{$\bullet$}}}

\newcommand{\BRD}{{\rm BRD}}

\renewcommand{\gg}{{\mathfrak g}}
\newcommand{\zz}{{\mathfrak z}}
\newcommand{\tl}{{\mathfrak t}}

\newcommand{\nn}{{\langle\nu\rangle}}

\newcommand{\SSS}{{\bf S}}

\begin{document}

\title[Real points in a real homogeneous space]%
{Real points in a homogeneous space\\ of a real algebraic group}

\author{Mikhail Borovoi}

\address{Raymond and Beverly Sackler School of Mathematical Sciences,
       Tel Aviv University, 6997801 Tel Aviv, Israel}

\email{borovoi@tauex.tau.ac.il}

\thanks{The author was partially supported
    by the Israel Science Foundation (grant No. 870/16)}



\date{\today}

\begin{abstract}
Let $\GG$ be a linear algebraic group over the field of real numbers $\R$,
and let $\YY$ be a right homogeneous space of $\GG$.
We wish to find a real point of $\YY$
or to prove that $\YY$ has no real points.
We describe a method to do that,
implicitly using second nonabelian Galois cohomology.
Our method is suitable for computer-assisted calculations.
\end{abstract}

\begin{comment}  Abstract for arXiv:

Let G be a linear algebraic group over the field of real numbers R,
and let Y be a right homogeneous space of G.
We wish to find a real point of Y
or to prove that Y has no real points.
We describe a method to do that,
implicitly using second nonabelian Galois cohomology.
Our method is suitable for computer-assisted calculations.

\end{comment}

\maketitle

\tableofcontents

\setcounter{section}{-1}

\section{Introduction}
\label{s:intro}

\begin{subsec}
We denote by $\R$ and $\C$ the fields
of real and complex numbers, respectively.
We write $\Gamma=\Gal(\C/\R)=\{1,\gamma\}$
for the Galois group of $\C$ over $\R$,
where $\gamma$ is the complex conjugation.

Let $\GG$ be a real linear algebraic group,
that is, an affine group scheme of finite type over $\R$.
We say just that $\GG$ is an $\R$-group.
Let $\GG(\R)$ and $\GG(\C)$ denote the groups of real points
and complex points of $\GG$, respectively.
We denote by  $G$
(the same letter, but not boldface)
the base change $G=\GG\times_\R\C$ of $\GG$ from $\R$ to $\C$.
By abuse of notation we identify $G$ with $\GG(\C)$.
In particular, $g\in G$ will mean $g\in\GG(\C)$.
The Galois group $\Gamma$ acts on $G$, that is, $\gamma$
acts on $G$ by an anti-holomorphic involution
\[\tau_\gG\colon G\to G,\quad g\mapsto \upgam g\ \  \text{for}\ g\in G.\]
Moreover, the anti-holomorphic involution $\tau_\gG$ is anti-regular
in the sense of Section \ref{s:anti} below.
Conversely, a pair $(G,\tau)$,
where $G$ is a $\C$-group (a linear algebraic group over $\C$)
and $\tau\colon G\to G$ is an anti-regular involution of $G$,
by Galois descent comes from a unique
(up to a canonical isomorphism) $\R$-group $\GG$;
see  references in Subsection \ref{ss:G-anti} below.
Then by abuse of notation we write $\GG=(G,\tau)$.
\end{subsec}

\begin{subsec}
Let $\YY$ be a right homogeneous space of $\GG$.
This means that $\YY$ is a quasi-projective $\R$-variety
(a reduced quasi-projective scheme of finite type over $\R$),
and if we write $Y=\YY\times_\R\C$
for the base change of $\YY$ from $\R$ to $\C$
and identify $Y$ with the set of $\C$-points $\YY(\C)$,
then we have a  {\em transitive} right $G$-action
on the quasi-projective $\C$-variety $Y$
\[ Y\times G\to Y, \quad (y,g)\mapsto y\cdot g\ \, \text{for}\ y\in Y,\,g\in G,\]
compatible with the $\Gamma$-actions on $G$ and on $Y$.
Here ``compatible'' means  that
\[\upgam\hs(y\cdot g)=\upgam y\cdot\hm\upgam g\quad\text{for all}\ y\in Y,\, g\in G.\]
\end{subsec}

\begin{subsec}
The group of $\R$-points of $\GG$
\[\GG(\R)=\{g\in G\mid \upgam g=g\} \]
naturally acts on the right on the set of $\R$-points of $\YY$
\[\YY(\R)=\{y\in Y\mid \upgam y =y\}, \]
and we wish to compute the orbits of this action
(we call them ``real orbits'').
In other words, we wish to compute the number of real orbits
(which is always finite) and to find representatives of all real orbits.

If we know one real point $y_1\in Y(\R)$,
then we can easily compute representatives of {\em all} real orbits
using first Galois cohomology.
See Serre \cite[Section I.5.4, Corollary 1 of Proposition 36]{Serre},
and also \cite[Proposition 3.6.5]{BGL21}.
\end{subsec}

\begin{subsec}\label{ss:setup}
Now assume that we do not know any real point in $Y$,
and even do not know whether $Y$ has a real ($\gamma$-fixed) point.
We choose once and for all a $\C$-point
$y_0\in Y$ and write $H={\rm Stab}_G(y_0)$.

We have
\[\upgam y_0=y_0\cdot \gyo\]
for some $\gyo\in G$, because $Y$ is a homogeneous space of $G$.
Then we have
\begin{align}
&\gyo\hm \upgam\gyo \in H,\label{e:gy00}\\
&\gyo\hm \upgam a\hs \gyo ^{-1}\in H \quad\text{for all}\ a\in H;\label{e:action-00}
\end{align}
see Subsection \ref{ss:Y-to-G,gyo} below.
Conversely, consider a pair $(H,\gyo)$ where $H\subset G$ is an algebraic $\C$-subgroup
and $\gyo\in G$ is an element. If $(H,\gyo)$ satisfies \eqref{e:gy00} and \eqref{e:action-00},
then it comes from some right homogeneous space $\YY$ of $\GG$ and a point $y_0$ of $Y$.
See Subsection \ref{ss:H,gyo-to-YY} below.
\end{subsec}

We come to our main problem:

\begin{subsec}{\bf Main Problem.}
\label{ss:Main}
For a right homogeneous space $\YY$ of $\GG$
given by a pair $(H,\gyo)$ as above
satisfying \eqref{e:gy00} and \eqref{e:action-00},
determine whether $\YY$ has a real point,
and if yes, give a method of finding such a point,
suitable for computer-assisted calculations.
\end{subsec}

In this article we solve Main Problem \ref{ss:Main}.
The plan of the article will be given in the next section.

\begin{subsec}{\bf Motivation.}
The motivation for Main Problem \ref{ss:Main}
came from our preprint \cite{BGL21},
where we classified $\SL(9,\R)$-orbits
in the space of real trivectors $\bigwedge^3\R^9$.
We used the classification of  $\SL(9,\C)$-orbits
in the space of {\em complex} trivectors $\bigwedge^3\C^9$
due to Vinberg and Elashvili \cite{VE78}.
For any $\SL(9,\C)$-orbit $\Oo\subset \bigwedge^3\C^9$
preserved by the complex conjugation,
we classified $\SL(9,\R)$-orbits in the set of real points
\[\Oo(\R)=\Oo\cap \bigwedge^3\R^9=\{x\in \Oo\mid\upgam x=x\}.\]
Any such orbit $\Oo$ is a left homogeneous space of $\GG=\SL_{9,\R}$.
There were some orbits
preserved by the complex conjugation in which
we could not find any real points by ad hoc methods.
We developed a method of finding a real point using
second Galois cohomology, which permitted us to find
real points in these orbits using computer.
In our examples we had $H^1\hs\GG=\{1\}$
and the stabilizer $H$ was abelian,
and for this reason, in \cite{BGL21}
we developed a method under these two assumptions only.
In the present article we develop a method
of solving Main Problem \ref{ss:Main} in the general case
(without these two assumptions).
We expect to use this method in our future classification projects.
\end{subsec}

\begin{subsec}
{\bf Notation and conventions}\label{not-conv}
\begin{itemize}
\item[\cc] $\Z$ denotes the ring of natural numbers.
\item[\cc] $\Q$, $\R$ and $\C$ denote the fields
           of rational numbers, of real numbers, and of complex numbers, respectively.
\item[\cc] $\Gamma=\Gal(\C/\R)=\{1,\gamma\}$, the Galois group of $\C$ over $\R$,
           where $\gamma$ is the complex conjugation.
\item[\cc] By an algebraic group, we always mean a {\em linear} algebraic group.
\item[\cc] We denote real algebraic groups and real algebraic varieties
           by boldface letters $\GG$, $\HH$, $\YY$, \dots,
           and their complexifications by respective Roman (non-bold) letters
           $G=\GG\times_\R \C$, $H=\HH\times_\R\C$, $Y=\YY\times_\R\C$, \dots
\item[\cc] By a $\C$-group we always mean an {\em algebraic group over $\C$}.
           Similarly, by an $\R$-group we always mean an {\em algebraic group over $\R$}.
\item[\cc] $\Aut H$ denotes the automorphism group of a $\C$-group $H$.
\item[\cc] $\Inn H$ denotes the group of inner automorphisms of $H$.
\item[\cc] $\Out H=\Aut H/\Inn H$,  the group of outer automorphisms of $H$.
\item[\cc] $\SAut H$ denotes the group of {\em semi-linear} automorphisms
           of a $\C$-group $H$, that is, of regular automorphisms and anti-regular ones;
           see Subsection \ref{ss:SAut-SOut} below.
\item[\cc] $\SOut H=\SAut H/\Inn H$.

\item[\cc] $\inn(g)\colon x\mapsto gxg^{-1}$ denotes the inner automorphism
           of a group $G$ corresponding to an element $g$ of $G$.
\item[\cc] $Z(\GG)$ denotes the center of $\GG$.
\item[\cc] $\X^*(T)=\Hom(T,\GmC)$, the character group of a torus $T$,
            where $\GmC$ denotes the multiplicative group over $\C$.
\item[\cc] $\X_*(T)=\Hom(\GmC,T)$, the cocharacter group of a torus $T$.
%           where $\GmC$ denotes the multiplicative group over $\C$.
\item[\cc] By an involution (of an algebraic group, of a variety, etc.) we mean
    an automorphism (regular or anti-regular) with square identity.
\item[\cc] When $\GG$ is an $\R$-group,
           $Z^1\hs\GG=\{z\in G\mid z\cdot\hm\upgam z=1\}$, \,\\
           $B^1\hs\GG=\{g^{-1}\cdot\hm\upgam g\mid g\in G\}$, \,
           $H^1\hs\GG:=H^1(\R,\GG)$.
\item[\cc] When $\CC$ is a commutative $\R$-group,
           $Z^2\hs\CC=\CC(\R)$, \,$B^2\hs\CC=\{c\cdot\hm\upgam c\mid c\in C\}$, \,
           $H^2\hs\CC:=H^2(\R,\CC)=Z^2\hs\CC/B^2\hs\CC$.
\item[\cc] $[d]$ denotes the cohomology class of a cocycle $z$.
\end{itemize}
\end{subsec}

\section{Plan of solving Main Problem \ref{ss:Main}}
\label{s:Plan}

\begin{subsec}\label{ss:Y-to-G,gyo}
Let $\GG$, $\YY$, $y_0\in Y$, $H=\Stab_G(y_0)\subset G$,
and $\gyo\in G$ be as in the introduction.
In particular, we have
\[\upgam y_0=y_0\cdot \gyo \hs.\]
We have
\[\upgam\hs(\upgam y_0)=y_0\hs,\]
hence
\[y_0=\upgam\hs(\upgam y_0)=\upgam\hs(y_0\cdot \gyo )=\upgam y_0\cdot\hm\upgam \gyo
           =y_0\cdot \gyo \cdot\hm\upgam \gyo \hs,\]
whence $\gyo\hs\upgam \gyo\in H$, which gives \eqref{e:gy00}.

If $\gyo' =a \gyo $ for some $a\in H$, then
\[\upgam y_0=y_0\cdot \gyo' \hs,\]
whence
\[\gyo' \cdot\hm\upgam \gyo' \in H,\]
that is,
\begin{align*}
&a\hs \gyo\hm\hm\upgam a \upgam \gyo \in H,\\
&a\cdot \gyo\hm\hm \upgam a\hs \gyo ^{-1}\cdot \gyo\hm\hm \upgam \gyo \in H.
\end{align*}
Since $a\in H$ and $\gyo\hm\hm \upgam \gyo \in H$,
we obtain that
\begin{equation*}
\gyo\hm\hm \upgam a\hs \gyo ^{-1}\in H,
\end{equation*}
which gives \eqref{e:action-00}.
Thus from a right homogeneous space $\YY$ of $\GG$
we obtain a pair $(H,\gyo)$
satisfying \eqref{e:gy00} and \eqref{e:action-00}.
\end{subsec}

\begin{subsec}\label{ss:H,gyo-to-YY}
Conversely, if we have a pair $(H,\gyo)$ as above
satisfying \eqref{e:gy00} and \eqref{e:action-00},
we may set $Y=H\backslash G$ and define
an anti-holomorphic automorphism $\tau_\yY$ of $Y$ by
\[\tau_\yY(Hg)=H \gyo\cdot\hm\upgam g\quad\text{for}\ g\in G.\]
If we take $g'=ag$ for some $a\in H$,
then we obtain
\[\tau_\yY(Hg')=H\cdot \gyo\cdot\hm\upgam (ag)
         =H\cdot \gyo\hm\hm\upgam a\hs \gyo^{-1}\cdot\gyo\hm\hm \upgam g
         =H \gyo\cdot\hm\upgam g=\tau_\yY(Hg)\]
because $\gyo\hm\hm\upgam a\hs \gyo^{-1}\in H$ by \eqref{e:action-00},
 whence we see that $\tau_\yY$ is well defined.
Moreover, for any $g''\in G$ we clearly have
\begin{equation}\label{e:compatible}
\tau_\yY(Hg\cdot g'')=\tau_\yY(Hg)\cdot\hm\upgam g''.
\end{equation}
Finally,
\[\tau_\yY^2(H\cdot 1)=\tau_\yY(H\cdot \gyo)=H\cdot \gyo\cdot\hm\upgam\gyo=H\cdot 1\]
because $\gyo\hs\upgam\gyo \in H$ by \eqref{e:gy00},
and since the automorphism $\tau_\yY$ satisfies \eqref{e:compatible},
we conclude that $\tau_\yY^2=\id_Y$.
It remains to note that the anti-holomorphic involution $\tau_\yY$ of $Y$
is anti-regular in a suitable sense,
and by Galois descent it defines a real form $\YY$ of $Y$
(we do not give details because we shall not use this fact).
By abuse of notation we write $\YY=(Y,\tau_\yY)$.
\end{subsec}

We wish to find an $\R$-point $\yo\in \YY(\R)$
or to show that there is no such point.

\begin{proposition}\label{p:R-point}
A $\C$-point $\yo=y_0\cdot\gz^{-1}\in Y$ with $\gz\in G$
is an $\R$-point if and only if
\begin{equation}\label{e:g01}
\gz^{-1}\cdot\hm\upgam\gz\in H \gyo \hs.
\end{equation}
\end{proposition}

\begin{proof}
Assume that $\yo\in\YY(\R)$, that is,  $\upgam\yo=\yo$.
Then
\[y_0\cdot\gz^{-1}=\yo=\upgam \yo =\upgam\hs(y_0\cdot\gz^{-1})
     =\upgam y_0\cdot\hm\upgam\gz^{-1}=y_0\cdot \gyo \cdot\hm\upgam\gz^{-1}\hs,\]
whence
\[y_0\cdot\gz^{-1}\cdot\hm\upgam\gz=y_0\cdot \gyo \hs,\]
which implies \eqref{e:g01}.

Conversely, assume that \eqref{e:g01} holds.
Write
\[\gz^{-1}\cdot\hm\upgam\gz= h \gyo \quad\text{with } h\in H.\]
Then
\begin{align*}
\upgam \yo=\upgam\hs(y_0\cdot\gz^{-1})
    &=\upgam y_0\cdot\hm\upgam \gz^{-1}
   =y_0\cdot \gyo \cdot\hm\upgam \gz^{-1}\\
   &=y_0\cdot h\cdot \gyo \cdot\hm\upgam \gz^{-1}
   =y_0\cdot \gz^{-1}\cdot\hm\upgam\gz \cdot\hm\upgam \gz^{-1}
   =y_0\cdot \gz^{-1}=\yo.
\end{align*}
Thus $\yo\in \YY(\R)$.
\end{proof}

\begin{proposition}
Assume that $\yo=y_0\cdot \gz^{-1}$ is an $\R$-point,
so that \eqref{e:g01} holds.
Set
\begin{equation}\label{e:d}
z=\gz^{-1}\cdot\hm\upgam \gz.
\end{equation}
Then
\begin{equation}\label{e:Hg-gam}
z\in H\gyo\hs.
\end{equation}
Moreover,
\begin{align}
%&z\in Z^1\hs\GG, \label{e:cocycle}\\
%&[d]=1\in H^1\hs\GG,\quad\text{that is,}\quad
&z\in B^1\hs\GG\subset Z^1\GG.\label{e:coboundary}
\end{align}
\end{proposition}

\begin{proof}
The assertion \eqref{e:Hg-gam} follows from \eqref{e:g01},
and the assertion %\eqref{e:cocycle} and
\eqref{e:coboundary} follows from \eqref{e:d}.
\end{proof}

\begin{subsec}
Now our problem of finding an $\R$-point of $\YY$
decomposes into two subproblems:
\begin{enumerate}\item[]
\begin{enumerate}\item[]
\begin{enumerate}
 \item[{\bf Problem A.}] Try to find $z\in H\gyo\cap Z^1\hs\GG$.
 \item[{\bf Problem B.}] Try to find $h\in H$ such that if we write  $z'=hz$, then \\
          $z'\in H\gyo\hs\cap\hs B^1\hs\GG$.
\end{enumerate}
\end{enumerate}
\end{enumerate}
Accordingly, our method of finding
an $\R$-point of $\YY$ consists of two steps:
in Step A we try to solve Problem A,
and in Step B we try to solve Problem B.

If we succeed to perform Steps A and B,
then we obtain $z'\in H\gyo$ such that
$z'=\gz^{-1}\cdot \hm\upgam \gz$ for some $\gz\in G$,
and we obtain a point $y_1:=y_0\cdot \gz^{-1}$,
which is real  by Proposition \ref{p:R-point}.
If not, then we conclude that $\YY$ has no $\R$-points.

In the rest of the article,
we describe our method of finding an $\R$-point of $\YY$.
First we describe Step B.
Concerning Step A, we construct an extension of abstract groups
\[ 1\to H\to E\to \Gamma\to 1,\]
where $H$ denotes the group of $\C$-points $H(\C)$, with the property
that the  elements of the preimage $E_\gam$ of $\gam$ in $E$,
when acting on $H$ by conjugation,
act by {\em anti-regular} automorphisms;
see Section \ref{s:anti} below.
Such extensions are related to the second nonabelian
Galois cohomology with coefficients in $H$.
In Step A we try to solve the following Problem A$'$:
find an element $x_\spl\in E_\gam$ such that $x_\spl^2=1$.
We divide Step A  into three substeps: A.1, A.2, and A.3.
In Step A.1 we consider the case when $H$ is finite,
in Step A.2 we consider the case when $H$ is connected reductive,
and in Step A.3 we consider the case when $H$ is unipotent.
The most difficult is Step A.2; we describe it in the last section.
\end{subsec}

\begin{subsec} The plan of the rest of the article is as follows.
In Section \ref{s:anti} we discuss the notion
of an anti-regular automorphism.
In Section \ref{s:B} we describe our method of solving Problem B
assuming that Problem A has been already solved.
In Section \ref{s:extensions} we state our Problem A
in terms of splittings of a certain extension.
In Section \ref{s:H2} we give the definition
of nonabelian $H^2$ over $\R$ following Springer \cite{Springer},
both in terms of extensions and in terms of cocycles.
In Section \ref{s:A'} we divide solving Problem A\hm$'$
(that includes Problem A) into three substeps: A.1, A.2, and A.3,
and describe our methods for Steps A.1 and A.3.
In Section \ref{s:A2}, the most complicated in the article,
we describe our method for Step A.2.
In Appendix \ref{App:PR-fundam} we prove Lemma \ref{t:fundam}
describing the structure of a fundamental torus
of a simply connected semisimple $\R$-group.
\end{subsec}

\section{Anti-regular maps}
\label{s:anti}

Here we discuss anti-regular maps of complex affine varieties
and anti-regular automorphisms of complex  algebraic groups.
For the notion of a semi-linear morphism of schemes over an arbitrary field
see \cite{Borovoi20}.
In this section,  we write $X(\C)$ (and not just $X$)
for the set of $\C$-points of a $\C$-variety $X$.

\begin{definition}
Let $X$ and $Y$ be affine varieties over $\C$.
We say that a map on $\C$-points
$\varphi\colon X(\C)\to Y(\C)$ is {\em regular},
if it comes from a morphism of varieties $X\to Y$.
We say that $\varphi\colon X(\C)\to Y(\C)$ is {\em anti-regular},
if for any regular function $f\colon Y(\C)\to\C$,
its inverse image $\varphi^*\hm f\colon X(\C)\to \C$ given by
\[ (\varphi^*\hm f)(x)=\overline{f(\varphi(x))} \]
is regular, where the bar denotes the complex conjugation.
\end{definition}

\begin{lemma}[easy]
\label{l:reg-anti-reg}
\begin{enumerate}
\item[(i)] The composition of two regular maps,
and the composition of two anti-regular maps, are regular.
\item[(ii)] The composition of a regular map and an anti-regular map,
in any order, is anti-regular.
\end{enumerate}
\end{lemma}

\begin{remark}
Any anti-regular map is anti-holomorphic,
but there exist anti-holomorphic maps that are not anti-regular.
For example, the homomorphism of real Lie groups
\[\G_a(\C)=\C\,\longrightarrow\,\C^\times=\G_m(\C),\quad x\mapsto \exp\,\ov x\]
is anti-holomorphic, but not anti-regular.
\end{remark}

\begin{subsec}
Let $\XX$ be a {\em real} affine variety.
In the coordinate language, the reader may regard $\XX$
as an algebraic subset in $\C^n$ (for some positive integer $n$)
defined by polynomial equations with {\em real} coefficients.
More conceptually, the reader may assume that $\XX$
is a reduced affine scheme of finite type over $\R$.
With any of these two equivalent definitions,
$\XX$ defines a covariant functor
\[A\rightsquigarrow \XX(A)\]
from the category of commutative unital $\R$-algebras to the category of sets.
Applying this functor to the $\R$-algebra $\C$ and the morphism of $\R$-algebras
\[\gamma\colon \C\to \C,\quad z\mapsto\ov z\ \ \text{for}\ z\in\C,\]
we obtain a set (a complex analytic space) $\XX(\C)=X(\C)$ together with a map
\begin{equation}\label{e:tau-X}
\tau_\xX=\XX(\gamma)\hs\colon\hs X(\C)\to X(\C),
\end{equation}
where we write $X=\XX\times_\R\C$, the base change of $\XX$ from $\R$ to $\C$.
By functoriality we have $\tau_\xX^2=\id$,
and by Lemma \ref{l:anti} below the map $\tau_\xX$ is anti-regular.
We say that $\tau_\xX$ is an {\em anti-regular involution} of $X$.
\end{subsec}

\begin{lemma}\label{l:anti}
For a real affine variety $\XX$,
the map \eqref{e:tau-X} defined above is anti-regular.
\end{lemma}

\begin{proof}
We may and shall assume that $\XX$
is embedded into the real affine space $\A_\R^n$
for some positive integer $n$.
For a complex point $x\in X(\C)\subset \C^n$
with coordinates $(x_1,\dots,x_n)$,
the point $\tau_\xX(x)$ has coordinates $(\ov x_1,\dots,\ov x_n)$.
A regular function $f$ on $X$ is  the restriction to $X$ of a polynomial $P$
in the coordinates $x_i$ with certain coefficients $c_\alpha$.
An easy calculation shows that $\tau_\xX^*f$ is the restriction to $X$
of the complex conjugate polynomial $\upgam P$
(the polynomial with coefficients $\upgam c_\alpha$),
hence a regular function on $X$.
It follows that the map $\tau_\xX$ is anti-regular, as required.
\end{proof}

\begin{subsec}\label{ss:G-anti}
Let $\GG$ be a real algebraic group.
As above, it defines a complex algebraic group $G=\GG\times_\R\C$
and an anti-regular  group automorphism
\[\tau_\gG=\GG(\gam)\hs\colon\hs G(\C)\to G(\C)\]
such that $\tau_\gG^2=\id$;
see, for instance, \cite[Section 1.1]{BT21}.
We say that $\tau_\gG$ is an {\em anti-regular involution} of $G$.
Thus from $\GG$ we obtain a pair $(G,\tau_\gG)$.

Conversely, by Galois descent any pair $(G,\tau)$, where $G$
is a {\em complex} algebraic group
and $\tau\colon G(\C)\to G(\C)$ is an anti-regular involution of $G$,
comes from a unique (up to a canonical isomorphism)
{\em real} algebraic group $\GG$;
see Serre \cite[V.4.20,  Corollary 2 of Proposition 12]{Serre-AGCF},
or the book ``N\'eron models'' \cite[Section 6.2, Example B]{BLR},
or Jahnel \cite[Theorem 2.2]{Jahnel}.
We shall not use this fact.
For us, a real algebraic group is a pair $(G,\tau)$ as above,
and we write $\GG=(G,\tau)$.
We say that the real algebraic group $(G,\tau)$
is a {\em real form} of the complex algebraic group $G$.

Note that if $G$ is reductive % (connected or not)
or unipotent,
then any anti-holomorphic involution of $G$ is anti-regular;
see Cornulier  \cite{Cornulier}.
The hypothesis that $G$ is either reductive or unipotent, is necessary:
the commutative algebraic group $\G_{a,\C}\times\G_{m,\C}=\C\times\C^\times$
has the  anti-holomorphic involution
$(z,w)\mapsto (\ov{z},\exp(i\ov{z})\ov{w})$
that is not anti-regular.
\end{subsec}

\section{Step B of solving Main Problem \ref{ss:Main}}
\label{s:B}

\begin{subsec}
We describe Step B.
Assume that we have already found $z\in H\gyo$
such that $z\in Z^1\hs\GG$, that is,
\[z\cdot\hm\upgam z=1.\]
We wish to find $z'\in H \gyo $ such that
\begin{align}
&z'\cdot\hm\upgam z'=1\quad\text{and} \label{e:z'}\\
&[z']=1\in H^1\hs\GG, \notag %  \label{e:z'-bis}
\end{align}
that is,
\[z'\in B^1\hs\GG.\]
We may write
\begin{equation}\label{e:hz}
z'=hz\quad \text{for some}\ h\in H.
\end{equation}
Then \eqref{e:z'} gives
\[h\hs z\hs\hs\upgam h\upgam z=1,\]
whence
\begin{equation}\label{e:d-H}
h\cdot z\hs\hs\upgam h\hs\hs z^{-1}=1
\end{equation}
(because $\upgam z=z^{-1}$).

Consider the anti-regular automorphism
\[\varphi_z\colon G\to G,\quad g\mapsto z\cdot\hm\upgam g \cdot z^{-1}.\]
We consider the twisted $\R$-group $_z\GG=(G,\varphi_z)$.
Since $z\in H\gyo$, we have $\varphi_z(H)=H$; see \eqref{e:action-00}.
Consider the automorphism
\[\varphi_z|_H\colon H\to H.\]
By abuse of notation we denote by $_z\HH$ the pair $(H,\varphi_z|_H)$
(we abuse the notation because there is no $\HH$, only $_z \HH$).
Then we have an inclusion of $\R$-groups $i\colon _z\HH\into\hs_z\GG$.
\end{subsec}

\begin{proposition}
Consider the composite map on cocycles
\begin{equation}\label{e:composite-cocycles}
t_z\circ i\colon Z^1\hm_z\HH\to
                Z^1\hm_z\GG\to Z^1\hs\GG,\quad h\mapsto h\mapsto hz
\end{equation}
and the induced composite map on cohomology
\begin{equation}\label{e:composite}
\mathcal{T}_z\circ i_*\colon  H^1\hm_z\HH\to
             H^1\hm_z\GG\to H^1\hs\GG,\quad [h]\mapsto [h]\mapsto [hz].
\end{equation}
There exists $z'\in H\gyo\cap B^1\hs\GG$
if and only if the image of the map \eqref{e:composite}
contains the neutral element $[1]\in H^1\hs\GG$.
\end{proposition}

\begin{proof}
Assume that there exists  $z'\in H\gyo\cap B^1\hs\GG$.
Then  $z'\in H\gyo\cap Z^1\hs\GG$.
Write $z'=hz$ for some $h\in H$.
Formulas \eqref{e:hz} and \eqref{e:d-H}
mean that the cocycle  $z'$
is contained in the image of the  composite map \eqref{e:composite-cocycles},
and hence the cohomology class $[z']$
is contained in the image of the  composite map \eqref{e:composite};
see Serre \cite[Section I.5.3, Proposition 35 bis]{Serre}
for a description of the twisting  map $\mathcal{T}_z$.
Since $z'\in B^1\hs\GG$, we have $[z']=[1]\in H^1\hs\GG$,
and we conclude that the image of the map \eqref{e:composite}
contains $[1]$, as required.

Conversely, assume that the image
of the map \eqref{e:composite} contains $[1]$.
Then there exists $z'\in B^1\hs\GG$ such that
$z'=hz$ for some $h\in H$, where $z\in H\gyo$.
It follows that  $z'\in H\gyo\cap B^1\hs\GG$, as required.
\end{proof}

\begin{subsec}
We can explicitly compute the finite set $H^1\hs\GG$
in the following sense.
We can find a (finite) set of cocycles $g_i\in Z^1\hs\GG$
representing all cohomology classes,
and for any cocycle $z\in Z^1\hs\GG$ we can determine,
to which of the cocycles $g_i$ it is cohomologous, and what is the element $b\in G$
such that $b^{-1}\cdot g_i\cdot\hm\upgam b=z$.
In particular, for any $z\in Z^1\hs\GG$ we can determine
whether $z\in B^1\hs\GG$, and if yes, then
we can find an element $\gz\in G$ such that $z=\gz^{-1}\hs\upgam\gz$.
%The details will be given in \cite{BdG}.
Similarly, we can explicitly compute
the finite set $H^1\hs_z\HH$, in particular,
we can find a finite set of cocycles $h_j\in  Z^1{}_z\HH$
representing all cohomology classes.
\end{subsec}

\begin{subsec}{\bf Method of calculation. }
For each cocycle $h_j\in  Z^1{}_z\HH$ as above,
we compute $z_j:= h_j\cdot z\in Z^1\hs\GG$
and determine whether $z_j\in B^1\hs\GG$.
If for some $j$ we have $z_j\in B^1\hs\GG$,
then we write $z_j=\gz^{-1}\hs\upgam\gz$,
and we obtain an $\R$-point $y_1:=y_0\cdot \gz^{-1}\in\YY(\R)$.
If for all $j$ we have $z_j\notin B^1\hs\GG$,
then we conclude that $\YY$ has no $\R$-points.
\end{subsec}

\section{Extensions}
\label{s:extensions}

\begin{subsec}
We restate Problem A in terms of splittings of a certain extension.
Consider the semidirect product $G\rtimes \Gamma$ with multiplication law
\begin{equation}\label{e:mult}
(g,s)\cdot(g',s')=(g\cdot \ups g',ss')\quad\text{for }g,g'\in G,\ s,s'\in\Gamma.
\end{equation}
The group  $G\rtimes \Gamma$ acts on $Y$ on the left by
\[(g,s)\cdot y= \ups y\cdot g^{-1}\quad\text{for }g\in G,\ s\in\Gamma,\ y\in Y.\]
Let $E$ denote the stabilizer of $y_0$ in  $G\rtimes \Gamma$, that is,
\[E=\{(g,s)\in G\rtimes \Gamma\mid \ups y_0=y_0\cdot g\}.\]
Then $E$ is a group with multiplication law \eqref{e:mult}.
We have a canonical surjective homomorphism
\[\pi\colon E\to\Gamma,\quad (g,s)\mapsto s\]
with kernel
\[\{(h,1)\mid h\in H\}.\]
We identify $H$ with this kernel via the embedding
\[\iota\colon H\into E,\quad h\mapsto (h,1).\]
We obtain a group extension
\begin{equation}\label{e:extension}
1\to H\labelto \iota E\labelto \pi\Gamma\to 1.
\end{equation}
We write $E_\gam=\pi^{-1}(\gam)$.
The group $E$ acts on $H$ by conjugation.
\end{subsec}

\begin{lemma}
The conjugation action of $E$ on $H\cong\iota(H)$
has the following property:
\begin{equation}\label{p-ty:action}
\text{The elements of $E_\gam$ act on $H$
by {\emm anti-regular} automorphisms.}
\end{equation}
\end{lemma}

\begin{proof}
Let $x=(g,\gam)\in E_\gam$.
An easy calculation shows that $x^{-1}=(\upgam g^{-1},\gam)$.
We calculate:
\[x\cdot\iota(h)\cdot x^{-1}=(g,\gam)\cdot(h,1)
        \cdot(\upgam g^{-1},\gam)=(g\upgam h g^{-1}, 1).\]
We must show that the map
\[h\mapsto g\upgam h g^{-1}\colon\ H\to H\]
is anti-regular.
Since any regular function on $H$
is the restriction of some regular function on $G$,
it suffices to show that the map
\begin{equation}\label{e:123}
h\mapsto g\upgam h g^{-1}\colon\ H\to G
\end{equation}
is anti-regular.
But this last  map is a composition of the following three maps:
\begin{align*}
h\mapsto h\colon &H\into G;\\
g'\mapsto\upgam g'\colon &G\to G;\\
g'\mapsto g\hs g'g^{-1}\colon &G\to G,
\end{align*}
of which the second map is anti-regular by Lemma \ref{l:anti},
while the first one and the third one are clearly regular.
By Lemma \ref{l:reg-anti-reg}(ii) the map \eqref{e:123} is anti-regular, as required.
\end{proof}

\begin{definition}
Let $x\in G\rtimes \Gam$.
We say that $x$ is a {\em splitting element of $E$}
if $x\in E_\gam$ and $x^2=1$.
\end{definition}

Note that a splitting element $x$ defines a {\em splitting}
of the extension \eqref{e:extension},
that is, a homomorphism $\sigma\colon \Gam\to E$
such that $\pi\circ\sigma=\id_\Gam$.
Namely, we set $\sigma(1)=1_E,\ \sigma(\gam)=x$.
Conversely, a splitting $\sigma$ of \eqref{e:extension}
gives a splitting element $x=\sigma(\gam)$.

\begin{lemma}\label{l:d-gam-splitting}
Let $x=(z,\gam)\in G\rtimes \Gam$.
Then $x$ is a splitting element of $E$ if and only if
$z$ is a solution of Problem A, that is,
$z\in H\gyo\cap Z^1\hs\GG$.
\end{lemma}

\begin{proof}
The assertion $x\in E_\gam$ means that $z\in H\gyo$\hs.
We have
\[x^2=(z,\gam)^2=(z\cdot\hm\upgam z,\,1).\]
Thus the equality $x^2=1$ means that $z\cdot\hm\upgam z=1$, that is,
$z\in Z^1\hs\GG$.
\end{proof}

\section{Nonabelian $H^2$}
\label{s:H2}

\begin{subsec}\label{ss:SAut-SOut}
We explain the notion of nonabelian $H^2$ in Galois cohomology over $\R$.
We follow Springer \cite[1.13--1.17 and 2.1--2.6]{Springer}.

Let $H$ be an algebraic $\C$-group.
We denote by $\SAut(H)$ the group of semi-linear automorphisms of $H$,
that is, of regular automorphisms and anti-regular automorphisms.
We have a natural homomorphism
\begin{equation}\label{e:SAut-G}
\SAut(H)\to\Gamma
\end{equation}
sending the regular automorphisms to 1
and sending the anti-regular ones to $\gamma$.
We assume that the homomorphism \eqref{e:SAut-G} is surjective, that is,
$H$ admits an anti-regular automorphism.
Consider the group of inner automorphisms of $H$:
\[\Inn(H)\subset\Aut(H)\subset\SAut(H).\]
For $\phi\in \SAut(H)$ and $h\in H$ we have
\[\phi\circ\inn(h)\circ \phi^{-1}=\inn(\phi(h)),\]
and hence  $\Inn(H)$ is a normal subgroup of $\SAut(H)$.
We set
\[\SOut(H)=\SAut(H)/\Inn(H).\]
The surjective homomorphism \eqref{e:SAut-G}
induces a short exact sequence
\begin{equation}\label{e:SOut-G}
1\to \Out(H)\to \SOut(H)\labelto\varsigma\Gamma\to 1,
\end{equation}
where $\Out(H):=\Aut(H)/\Inn(H)$ is the group of outer automorphisms of $H$.

\begin{definition}\label{d:R-kernel}
An $\R$-kernel in a  $\C$-group $H$
is a splitting of the extension \eqref{e:SOut-G},
that is, a homomorphism $\kappa\colon \Gamma\to \SOut(H)$
such that $\varsigma\circ\kappa=\id_\Gamma$.
\end{definition}

For a given $\C$-group $H$, we consider extensions
\begin{equation}\label{e:E0}
1\to H\to E\to\Gamma\to 1%\tag{E}
\end{equation}
with property \eqref{p-ty:action},
where we write $H$ for the group of $\C$-points $H(\C)$.
From such an extension we obtain a homomorphism
\[E\to \SAut(H),\quad  x \mapsto\inn(x)|_H\hs\]
sending $H$ to $\Inn(H)$ and thus inducing an $\R$-kernel in $H$
\[\kappa\colon \Gamma\to \SOut(H),\]
which we call the {\em $\R$-kernel in $H$  associated with $E$.}

\begin{definition}[{Springer \cite[Definition 1.14]{Springer}}]
\label{d:H2}
For a given complex algebraic group $H$ and an $\R$-kernel $\kappa$ in $H$,
we denote by $H^2(H,\kappa)$ the set of isomorphism classes
of group extensions \eqref{e:E0} with property \eqref{p-ty:action}
whose associated $\R$-kernel is $\kappa$.
\end{definition}
\end{subsec}

We pass to the cocyclic description of nonabelian $H^2$.

\begin{construction}\label{con:ext-to-cocycles}
Let \eqref{e:E0} be an extension as in Definition \ref{d:H2}.
Let  $x\in E_\gam:=\pi^{-1}(\gam)\subset E$.
We set
\[\fx=\inn(x)|_H\in \SAut(H),\quad \hxx=x^2\in H.\]
Then
\begin{align*}
& \fx^2:=\inn(x^2)=\inn(\hxx), \\
& \fx(\hxx)=x\cdot x^2\cdot x^{-1}=x^2=\hxx\hs, \\
& \fx\cdot\Inn(H)=\kappa(\gam).
\end{align*}
If instead of $x\in E_\gam$ we choose another element $x'=ax$ with $a\in H$
and set  $\fx'=\inn(x')\in\SAut(H)$, $\hxx'=(x')^2\in H$,
then we obtain
\begin{align*}
&\fx'=\inn(ax)=\inn(a)\circ\inn(x)=\inn(a)\circ \fx,\\
&\hxx'=(ax)^2=axax=a\cdot xax^{-1}\cdot x^2=a\cdot \fx(a)\cdot \hxx\hs.
\end{align*}
\end{construction}

\begin{definition}
For a given $\R$-kernel $\kappa$ for $H$,
the set of 2-{\em cocycles} $Z^2(H,\kappa)$
is the set of pairs $(f,h)$,
where $f\in\SAut(H)$ is an anti-regular automorphism,
$h\in H$, and the following 2-cocycle conditions are satisfied:
\begin{enumerate}
\item[(i)] $f^2=\inn(h)$;
\item[(ii)] $f(h)=h$;
\item[(iii)] $f\cdot\Inn(H)=\kappa(\gamma)$.
\end{enumerate}
We say that two 2-cocycles $(f,h)$ and $(f',h')$ in $Z^2(H,\kappa)$
are {\em equivalent},
and write $(f,h)\sim(f',h')$, if
there exists $a\in H$ such that
\begin{equation*}
f'=\inn(a)\circ f,\quad  h'=a\cdot f(a)\cdot h.
\end{equation*}
\end{definition}

We have seen that Construction \ref{con:ext-to-cocycles} gives a map $(E,x)\mapsto(\fx,\hxx)$
from the set of isomorphism classes of pairs $(E,x)$
with associated $\R$-kernel $\kappa$
to the set of cocycles $Z^2(H,\kappa)$,
which induces a map from $H^2(H,\kappa)$ to $Z^2(H,\kappa)/\sim$\hs.

\begin{proposition}  [see Mac Lane {\cite[Lemma IV.8.2]{MacLane}}\hs]
The map of construction \ref{con:ext-to-cocycles} induces a bijection
$H^2(H,\kappa)\isoto Z^2(H,\kappa)/\sim$\hs.
\end{proposition}

\begin{proof}[Idea of proof]
We describe the inverse map to the map of Construction  \ref{con:ext-to-cocycles}.
For a 2-cocycle $(f,h)\in Z^2(H,\kappa)$, we set
\begin{gather*}f_1=\id_H\hs, \quad \fg=f\in \SAut(H),\\
h_{1,1}=h_{\gam,1}=h_{1,\gam}=1,\quad \hgg=h\in H.
\end{gather*}
We obtain maps
\begin{align*}
s\mapsto f_s\colon \,\Gamma\to \SAut(H), \qquad
           (s,s')\mapsto h_{s,s'}\colon\,\Gam\times\Gam\to H.
\end{align*}
To the equivalence class of our  2-cocycle  $(f,h)\in Z^2(H,\kappa)$
we associate the isomorphism class of the group extension
\[1\to H\labelto\iota E\labelto{\pi}\Gamma\to 1,\]
where $E=H\times\Gam$ with the multiplication law
\[(h,s)\cdot(h',s')=(h\cdot f_s(h')\cdot h_{s,s'}\hs, \,  s s')\]
and the  homomorphisms $\iota$ and $\pi$ are the obvious ones:
\begin{align*}
&\iota\colon H\into E,\quad h\mapsto (h,1),\\
&\pi\colon E\onto\Gamma,\quad (h,s)\mapsto s.\qedhere
\end{align*}
\end{proof}

\begin{remark}
We shall not use $\R$-kernels and $H^2(H,\kappa)$ in this article,
but we shall use the relation between the isomorphism classes of extensions
and the equivalence classes of 2-cocycles.
\end{remark}

\section{Problem A$'$, Steps A.1 and A.3}
\label{s:A'}

We consider the following problem:

\begin{subsec}
{\bf Problem A\hm$\boldsymbol{'}$.}
Let $H$ be a $\C$-group,  let
\begin{equation}\label{e:E}
1\to H\to E\to\Gamma\to 1\tag{$E$}
\end{equation}
be a group extension with property \eqref{p-ty:action},
and let $(f,h)$ be a corresponding 2-cocycle.
We wish to find a splitting element of $E$ or to show
that there are no splitting elements in $E$.
Recall that a splitting element of $E$ is
an element $x\in E_\gam$ such that $x^2=1$.
\end{subsec}

\begin{subsec}\label{ss:HH0}
By Lemma \ref{l:d-gam-splitting},
our Problem A for the homogeneous space $\YY$
is Problem A\hm$'$ for the extension  \eqref{e:extension}
corresponding to $\YY$.

We need only one splitting element of $E$ for Step B.
However, in order to find a  splitting element by d\'evissage,
we need {\em all} splitting elements of $E$.
We explain how to find all splitting elements of $E$
after we have found one splitting element $x_0$.
We explain also how to find all conjugacy classes
of splitting elements of $E$.

Let $x_0\in E_\gamma$ be a splitting element.
It defines a 2-cocycle $(\fxo ,1)$, where $\fxo=\inn(x_0)|_H$\hs.
We have $\fxo ^2=\inn(x_0^2)=\id_H$.
Thus we obtain a real form $\HH_0=(H,\fxo)$ of $H$.
\end{subsec}

\begin{proposition}\label{p:conjugacy}
For $x_0$ and $\HH_0$ as in \ref{ss:HH0}, the map
\[ H\to E, \quad a\mapsto ax_0\]
induces a bijection between $Z^1\hs\hs\HH_0$
and the set of splitting elements of $E$,
which in turn induces a bijection between $H^1\hs\hs \HH_0$
and the set of conjugacy classes of splitting elements of $E$.
\end{proposition}

\begin{proof}
Let $x\in E_\gamma$. Write $x=ax_0$ with $a\in H$.
Then
\[x^2=(ax_0)^2=a\cdot x_0 a x_0^{-1}=a\cdot \fxo(a).\]
Thus $x^2=1$ if and only if $a\in Z^1\hs\hs\HH_0$\hs.

Let $x\in  E_\gam$.
Since the group $E$ is generated by $x$ and $H$,
we see that the $E$-conjugacy class of $x$ coincides
with the $H$-conjugacy class of $x$.
Write $x=ax_0$ and $x'=a'x_0$ with $a,a'\in Z^1\hs\hs\HH$.
If $x'$ is $H$-conjugate to $x$, that is,
$x'=b^{-1}x\hs b$ for some $b\in H$, then
\[x'=b^{-1}ax_0b=b^{-1} a\hs \fxo (b)\cdot x_0\hs,\]
whence $a'=b^{-1} a\hs \fxo (b)$.
In other words $a'\sim a$ in $Z^1\hs\hs\HH_0$.
Conversely, if $a'\sim a$ in $Z^1\hs\hs\HH_0$, that is,
$a'=b^{-1} a\hs \fxo (b)$ for some $b\in H$, then
\[x'=a'x_0=b^{-1} a\hs \fxo (b)\cdot x_0=b^{-1}ax_0b=b^{-1}xb,\]
whence $x'$ is conjugate to $x$.
\end{proof}

In order to divide Problem $\rm A'$ into subproblems,
we introduce certain subquotients of $H$ and $E$.

\begin{notation}
For a $\C$-group $H$:

\begin{itemize}
\item[\cc] $H^\circ$ denotes the identity component of $H$,
   which is a connected $\C$-group;
\item[\cc] $H^\ff =H/H^\circ$, which is a finite group;
\item[\cc] $H^\uu=R_u(H^\circ)$, the unipotent radical of $H^\circ$,
   which is a unipotent $\C$-group;
\item[\cc] $H^\mmu=H/H^\uu$, which is a reductive $\C$-group, not necessarily connected;
\item[\cc] $H^\red=H^\circ/H^\uu$, which is a connected reductive $\C$-group;
\item[\cc] $H^\sss=(H^\red,H^\red)$, the commutator subgroup of $H^\red$,
   which is a semisimple $\C$-group;
\item[\cc] $H^\ssc$ is the universal cover of $H^\sss$,
   which is a simply connected semisimple $\C$-group;
\item[\cc] $\rho\colon H^\ssc\onto H^\sss\into H^\red$ is the composite homomorphism,
   which in general is neither injective nor surjective.
\end{itemize}
\end{notation}

\begin{subsec}
Note that $H^\uu$ and $H^\circ$ are characteristic subgroups of $H$,
and hence they are normal in $E$.
By taking quotients by $H^\uu$ and $H^\circ$,
we obtain the following extensions
from our extension \eqref{e:E} of Problem A\hm$'$\hs:
the extension
\begin{equation}\label{e:E-mmu}
1\to H^\mmu\to E^\mmu\to\Gamma\to 1,\tag{$E^\mmu$}
\end{equation}
where $E^\mmu=E/H^\uu$, and the extension
\begin{equation}\label{e:E-ff}
1\to H^\ff \to E^\ff \to\Gamma\to 1,\tag{$E^\ff $}
\end{equation}
where $E^\ff =E/H^\circ$.
We plan to try to find a splitting element of $E^\ff$,
then to try to lift it to a splitting element of $E^\mmu$,
and then to lift it to a splitting element of $E$.
\end{subsec}

\begin{subsec}\label{ss:Plan-A'}
We wish to solve Problem  A\hm$'$,
that is, to find a splitting element $x$ of $E$.
If $x_\spl$ is a splitting element of $E$,
then its image $x_\spl^\ff\in E^\ff$
is a splitting element of $E^\ff$.
Therefore, we are going first to find all splitting elements of $E^\ff$,
and after that to try to lift them to splitting elements of $E$.
Note that if $x^\ff_1$ and $x^\ff_2$ are two {\em conjugate}
splitting elements of $E^\ff$, and $x^\ff_1$
admits a lifting to a splitting element of $E$, then so does $x^\ff_2$.
Therefore, when trying to find a splitting element of $E^\ff$
that can be lifted to a splitting element of $E$,
it suffices to check one representative
in each conjugacy class of splitting elements of $E^\ff$.
\end{subsec}

\begin{subsec}
{\bf Step A.1: Splitting elements of $E^\ff $: finite algebraic groups.}
The group $H^\ff$ is finite, hence set $E^\ff_\gam$ is finite as well.
We try to find a splitting element
$x^\ff$ of $E^\ff$ by brute force, that is,
by squaring all elements of $E^\ff_\gam$\hs.
If there is no splitting elements in $E^\ff$,
we conclude that $E$ has no splitting elements either,
and hence $\YY$ has no real points.
If we find a splitting element $x^\ff_0$ of  $E^\ff $,
then we obtain a real form $\HH_0^\ff$ of $H^\ff$.
We compute $H^1\hs\hs\HH^\ff _0$ and choose representatives
$a_0^\ff =1,a_1^\ff ,\dots,a_{n^\ff}^\ff \in H^\ff $
of all cohomology classes.
We set $x_i^\ff =a_i^\ff \cdot x_0^\ff $ for $i=0,1,\dots,n^\ff$.
By Proposition \ref{p:conjugacy},
the obtained elements $x_i^\ff$ are representatives
of all conjugacy classes of splitting elements of $E^\ff$.
(Alternatively, we can find such representatives $x_i^\ff$ by brute force.)
\end{subsec}

\begin{subsec}
{\bf Step A.2: Lifting a splitting element to $E^\mmu$: connected reductive groups.}\,
For each $i=0,1,\dots, n^\ff$ we consider the preimage $E^\circ_i$ in $E$
of the subgroup $\{1,x^\ff_i\}\subset E^\ff$.
Then we obtain an extension
\begin{equation}\label{e:E-circ-i}
1\to H^\circ\to E^\circ_i\to\Gamma\to 1,\tag{$E^\circ_i$}
\end{equation}
The subgroup $H^\uu$ is normal in $E$ and hence in $E^\circ_i$.
We set $E^\red_i=E^\circ_i/H^\uu$ and obtain an extension
\begin{equation}\label{e:E-red-i}
1\to H^\red\to E^\red_i\to\Gamma\to 1,\tag{$E^\red_i$}
\end{equation}

For each $i=0,1,\dots,n^\ff$ we try to find
a splitting element  $x^\red_i$ of  $E^\red_i$.
This is a rather complicated calculation,
which we shall describe in Section \ref{s:A2} below.

If we can find such $x^\red_i$ for some $i$,
then $x^\mmu:=x^\red_i$ is a splitting element of $E^\mmu$,
and we pass to Step A.3.

If there is no such $x^\red_i$ in  $E^\red_i$ for all $i=0,1,\dots,n^\ff$,
then we conclude that  $E^\mmu$ and $E$ have no splitting elements,
and $\YY$ has no real points.
\end{subsec}

\begin{subsec}
{\bf Step A.3: Lifting a splitting element to $E$: unipotent groups.}
Assume that we already have a splitting element $x^\mmu\in E^\mmu_\gamma$,
and we wish to lift it to a splitting element $x_\spl\in E_\gamma$.

Let $E^\uu$ denote the preimage in $E$
of the subgroup  $\{1,x^\mmu\}\subset E$.
Then we obtain an extension
\begin{equation}\label{e:E-u}
1\to H^\uu\to E^\uu\to\Gamma\to 1,\tag{$E^\uu$}
\end{equation}
We wish to find a splitting element $x_\spl$ in $E^\uu$.
\end{subsec}

\begin{theorem}[\hs{Douai \cite[Theorem IV.1.3]{Douai}}\hs]
\label{t:Douai}
For any extension with property \eqref{p-ty:action}
\begin{equation*}\label{e:u}
1\to H\to E\to\Gamma\to 1,
\end{equation*}
where  $H$ is a {\emm unipotent} $\C$-group,
there exists a splitting element in $E$,
and all splitting elements are conjugate.
\end{theorem}

In \cite{Douai} this theorem was stated in terms of gerbes
and proved by induction on the dimension of $H$.
Our proof below gives a splitting element of $E$ in one step.

\begin{proof}[Proof of Theorem \ref{t:Douai}]
Let  $x\in E_\gam$.
We write
\[\fxu=\inn(x)|_{H}\hs,\quad \hxu=x^2\in H.\]
Then
\[ \fxu^2=\inn(\hxu),\quad \fxu(\hxu)=\hxu\hs.\]
By abuse of notation, we also write $\fxu$ for the differential
\[d\fxu\colon \Lie H\to \Lie H.\]
Then the polynomial maps
\[\exp\colon \Lie H\to H\quad\text{and}
      \quad \log\colon  H\to \Lie H\]
are $\fxu$\,-equivariant.
We set
\[r=\exp(\tfrac12\log \hxu)\in H.\]
Then \hs$r^2=\hxu$, \,$\fxu(r)=r$. \hs Set
\[x_0=r^{-1} x\in E_\gam\hs.\]
Then
\[x_0^2=r^{-1} x\hs r^{-1} x
    =r^{-1} \cdot \fxu(r^{-1})\cdot x^2=\hxu^{-1}\cdot \hxu=1.\]
Thus $x_0$ is a desired splitting element of $E$.

By Proposition \ref{p:conjugacy}, the set of conjugacy classes
of splitting elements of $E$ is in a canonical bijection
with $H^1\hs\hs\HH_0$, where $\HH_0=(H,\fxo)$.
Note that the real algebraic group $\HH_0$ is unipotent.
Hence $H^1\hs\hs\HH_0=1$
(\hs see Serre \cite[Section III.2.1, Proposition 6]{Serre}\hs)
and all splitting elements are conjugate.
\end{proof}

\section{Step A.2: connected reductive groups}
\label{s:A2}

\begin{subsec}\label{ss:red-ext}
Here $H$ is a connected reductive $\C$-group.
We have an extension
\begin{equation}\label{e:extension-red}
1\to H\to E\to\Gamma\to 1\tag{E}
\end{equation}
with property \eqref{p-ty:action}.
Let $x\in E_\gam$. We obtain a 2-cocycle $(\fx,\hxx)$,
where
\[\fx=\inn(x)|_H\hs\in\hs\SAut(H)\quad\text{and}\quad\hxx=x^2\in H.\]
We have
\[\fx^2=\inn(\hxx),\quad \fx(\hxx)=\hxx\hs.\]
We wish to find a splitting element $x_\spl$ of $E_\gam$
or to show that there is no such element.
In other words, we wish to find $a\in H$
such that $a\cdot \fx(a)\cdot \hxx=1$,
or to show that there is no such $a$.
If we find such $a$, then we set $x_\spl=ax$.

We write $C=Z(H)$ and $C^\ssc=Z(H^\ssc)$
for the centers of $H$ and $H^\ssc$, respectively.
Write $\tau_\Cc =\fx|_C\colon C\to C$.
Since $\fx^2$ is an {\em inner} automorphism of $H$,
we have $\tau_\Cc ^2=\id$.
Moreover, $\tau_\Cc $ is an anti-regular automorphism of $C$,
and it does not depend on the choice of $x\in E_\gam$\hs.
Thus we obtain a canonically defined real form $\CC=(C,\tau_\Cc )$ of $C$.

The anti-regular automorphism $\fx\in\hs\SAut(H)$
induces an anti-regular automorphism $\fx^\ssc\in\hs\SAut(H^\ssc)$.
Write $\tau_\Cc^\ssc =\fx^\ssc\hs|{}_{C^\ssc}\colon C^\ssc\to C^\ssc$.
As above we see that $\tau_\Cc^\ssc$ is an anti-regular involution of $C^\ssc$,
and we obtain a canonically defined real form
$\CC^\ssc=(C^\ssc,\tau_\Cc ^\ssc)$ of $C^\ssc$.
We have a canonical  $\R$-homomorphism $\rho\colon \CC^\ssc\to \CC$.

The plan of the rest of the article is as follows.
We construct a 2-cocycle
$(\fx',\hxx')=a\cdot (\fx,\hxx)$ for some $a\in H$
such that $\hxx'\in C$.
Then $\hxx'\in Z^2\hs\CC$.
We consider the cohomology class  $[\hxx']\in  H^2\hs\CC$.
We show that there exists a splitting element $x_\spl$ in $E$
if and only if
\[[\hxx']\in{\rm im}\big[\rho_*\colon H^2\hs\CC^\ssc\to H^2\hs\CC\big],\]
and in this case we construct such $x_\spl$.
\end{subsec}

\begin{construction}
We choose a maximal torus $T\subset H$
and a Borel subgroup $B$ of $H$
such that $T\subset B\subset H$.
We write $\hh=\Lie\hs H$ and write the root decomposition for $H$
\[ \hh=\Lie\hs T\oplus\bigoplus_{\beta\in R}\hh_\beta\hs,\]
where $R=R(H,T)$ is the root system,
and $\hh_\beta$ is the root subspace
corresponding to a root $\beta\in R$.
We write also the root decomposition for $B$
\[ \Lie\hs B=\Lie\hs T\oplus\!\bigoplus_{\beta\in R_+} \hh_\beta\hs,\]
where $R_+=R_+(H,T,B)$ is the set of positive roots corresponding to $B$.
Let $S=S(H,T,B)$ denote the set of simple roots,
that is, of positive roots $\alpha\in R_+$
that are not sums of two or more positive roots.
For any $\alpha\in S$ we choose a nonzero element $X_\alpha\in\hh_\alpha$.
We say that the  $(T,B,\{X_\alpha\}_{\alpha\in S})$ is a {\em pinning} of $H$;
see Conrad \cite[Definition 1.5.4]{Conrad}.
\end{construction}

\begin{construction}\label{constr:x'}
Let $x, \fx\hs,\hxx$ be as in \ref{ss:red-ext}.
Consider $\fx(T,B,\{X_\alpha\})$.
By Borel \cite[Theorem 11.1 and Corollary 11.3(1)]{Borel}
there exists $a\in H$ such that
\[a\cdot \fx(T)\cdot a^{-1}=T,\quad a\cdot \fx(B)\cdot a^{-1}=B.\]
After multiplying $a$ on the left by some $t\in T$, we may assume that
\[a\cdot \fx(T,B,\{X_\alpha\})\cdot a^{-1}=(T,B,\{X_\alpha\}).\]
Thus if we set $x'=ax$,  $\fx'=\inn(x')$, and $\hxx'=(x')^2$,
then $\fx'$ preserves the pinning $(T,B,\{X_\alpha\})$ of $H$.
It follows that $(\fx')^2=\inn(\hxx')$
is an {\em inner} automorphism of $H$
preserving the pinning $(T,B,\{X_\alpha\})$.
However, the only inner automorphism of $H$
preserving a pinning is the identity automorphism.
We see that $\inn(\hxx')=\id_H$ and hence $\hxx'\in C=Z(H)$.

We set $\tau_\Hh =\fx'\colon H\to H$; then $\tau_\Hh$
is an anti-regular automorphism of $H$,
and $\tau_\Hh^2=\id_H$.
We consider the real form $\HH=(H,\tau_\Hh)$ of $H$.
The restriction of $\tau_\Hh $ to $C$ is clearly $\tau_\Cc $.
Thus $\CC=Z(\HH)$ and $\hxx'\in Z^2\hs\CC$.
We consider the cohomology class
\[ [(\hxx')^{-1}]\in H^2\hs\CC.\]

Consider the short exact sequence
\[ 1\to \CC\to\HH\to\HH/\CC\to 1\]
and the corresponding cohomology exact sequence
\[ \cdots\to H^1\hs\CC\to H^1\hs\HH\to
            H^1\hs\hs\HH/\CC\labelto{\Delta} H^2\hs\CC.\]
\end{construction}

\begin{lemma}[well-known]
\label{e:D}
Consider an exact sequence of $\Gamma$-groups
\[1\to A\to B\to B/A\to 1,\]
where $A$ is a central $\Gamma$-subgroup of $B$,
and let
\[\Delta\colon H^1\hs B/A\to H^2\hm A\]
denote the coboundary map. Let $a\in Z^2\hm A$.
Then $[a]\in\im\Delta$ if and only if
$a=b\cdot\hm\upgam b$ for some $b\in B$.
\end{lemma}

\begin{lemma}\label{l:b}
With the notation of Construction \ref{constr:x'},
for any element $b\in H$, the element $bx'\in E_\gam$
is a splitting element if and only if
$(\hxx')^{-1}=b\cdot\tau_\Hh (b)$.
\end{lemma}

\begin{proof}
Write $x''=bx'$ with $b\in H$, and set $\hxx''=(x'')^2\in H$. Then
\[\hxx''=b\cdot \fx'(b)\cdot \hxx'=b\cdot\tau_\Hh (b)\cdot \hxx'.\]
The element $x''$ is a splitting element if and only if $\hxx''=1$
if and only if $(\hxx')^{-1}=b\cdot\tau_\Hh (b)$.
\end{proof}

\begin{corollary}\label{c:H-ad}
With the notation of Construction \ref{constr:x'},
the following assertions are equivalent:
\begin{enumerate}
\item[(i)] There exists a splitting element in $E$;
\item[(ii)] $(\hxx')^{-1}=b\cdot\tau_\Hh (b)$ for some $b\in H$;
\item[(iii)] $[(\hxx')^{-1}]\in \im \hs\Delta$.
\end{enumerate}
\end{corollary}

\begin{proof}
(i)$\Leftrightarrow$(ii) by Lemma \ref{l:b}, and
(ii)$\Leftrightarrow$(iii) by Lemma \ref{e:D}.
\end{proof}

\begin{proposition}\label{l:ssc}
Let $\HH$ be a {\emm simply connected} semisimple $\R$-group,
and write $\CC=Z(\HH)$.
Then the connecting map
$\Delta\colon \, H^1\hs\hs \HH/\CC\to H^2\hs \CC$ is surjective.
\end{proposition}

This result an be deduced from the following lemma:

\begin{lemma}[\hs{\cite[Lemma 6.18]{PR}}\hs]
\label{l:PR}
For $\HH$ as in Proposition \ref{l:ssc},
there exists a maximal torus $\TT\subset \HH$
such that $H^2\hs\hs\TT=1$.
\end{lemma}

For calculations we need a more precise result:

\begin{lemma}\label{l:PR-fundam}
For $\HH$ as in Proposition \ref{l:ssc},
let $\TT\subset \HH$ be a fundamental torus, that is,
a maximal torus containing a maximal compact torus.
Then
\begin{enumerate}
\item[(i)] The group $\TT(\R)$ is a direct product of groups
    isomorphic to $\C^\times$ or  $U(1)=\{z\in\C^\times \mid z\bar z=1\}$.
\item[(ii)]  $H^2\hs\hs\TT=1$.
\end{enumerate}
\end{lemma}

This lemma will be proved in Appendix \ref{App:PR-fundam}.
Note that assertions (i) and (ii) are equivalent; see Remark \ref{r:H2} below.

\begin{proof}[Proof of Proposition \ref{l:ssc}]
Let $\TT\subset \HH$ be a fundamental torus. Then $\CC\subset\TT$.
Consider the commutative diagram with exact rows
\[
\xymatrix{
1\ar[r]  &\CC\ar[r]\ar@{=}[d]  &\TT\ar[r]\ar[d]  &\TT/\CC\ar[r]\ar[d] &1\\
1\ar[r]  &\CC\ar[r]            &\HH\ar[r]        &\HH/\CC\ar[r] &1
}
\]
and the induced commutative diagram with exact top row
\[
\xymatrix{
H^1\hs\hs\TT/\CC\ar[r]^-{\Delta_T}\ar[d] &H^2\hs\CC\ar[r]\ar@{=}[d]  &H^2\hs\hs\TT\\
H^1\hs\hs\HH/\CC\ar[r]^-{\Delta}          &H^2\hs\CC
}
\]
By Lemma \ref{l:PR-fundam} we have $H^2\hs\hs\TT=1$.
Hence, the map $\Delta_T$ is surjective, and so is $\Delta$.
\end{proof}

\begin{theorem}\label{t:spl-rho}
With the notation of Construction \ref{constr:x'},
$E$ has  a splitting element if and only if
\[ [\hxx']\in\im\big[\rho_{\cC,*}\colon H^2\hs\CC^\ssc\to H^2\hs\CC\big].\]
\end{theorem}

\begin{proof}
The commutative diagram with exact rows
\[
\xymatrix{
1\ar[r] &\CC^\ssc\ar[r]\ar[d]%^{\rho_\cC}
                             &\HH^\ssc\ar[r]\ar[d]^\rho  &\HH^\ssc/\CC^\ssc\ar[r]\ar[d]^\cong &1 \\
1\ar[r] &\CC\ar[r]                 &\HH\ar[r]                  &\HH/\CC\ar[r]           &1
}
\]
gives rise to a commutative diagram
\begin{equation}\label{e:Delta-Delta-sc}
\begin{aligned}
\xymatrix@C=40pt{
H^1\hs\hs\HH^\ssc/\CC^\ssc \ar[r]^-{\Delta^\ssc}\ar[d]_\cong  &H^2\hs\CC^\ssc\ar[d]%^{\rho_{\cC,*}}
\\
H^1\hs\hs\HH/\CC \ar[r]^-\Delta                                &H^2\hs\CC
}
\end{aligned}
\end{equation}
in which the right-hand vertical arrow is  $\rho_{\cC,*}$.

By  Corollary \ref{c:H-ad}, there exists a splitting element in $E$
if and only if $[(\hxx')^{-1}]\in\im\Delta$.
By Proposition \ref{l:ssc}, the map $\Delta^\ssc$ in diagram \eqref{e:Delta-Delta-sc} is surjective,
and we see from the diagram that $[(\hxx')^{-1}]\in\im\Delta$
if and only if $[(\hxx')^{-1}]\in\im\rho_{\cC,*}$\hs.
Since $\im\rho_{\cC,*}$ is
a {\em subgroup} of the abelian group $H^2\hs\CC$, we have
\[ [(\hxx')^{-1}]\in\im\rho_{\cC,*}\quad \text{if and only if}\quad [\hxx']\in\im\rho_{\cC,*}\hs,\]
which completes the proof.
\end{proof}

\begin{subsec}
The $\R$-group $\CC$ is a {\em quasi-torus}, that is,
it is isomorphic to the kernel of a homomorphism of $\R$-tori.
We can explicitly compute $H^2\hs\CC$.
This means that we can find representatives  $z_i\in Z^2\hs\CC$
of all cohomological classes, and that we have an algorithm
permitting us, for each 2-cocycle $z\in Z^2\hs\CC$,
to determine  to which $z_i$ it is cohomologous,
and giving an element $c\in C$ such that $z=z_i\cdot c\cdot\hm\upgam c$.
In particular, our algorithm permits us to determine whether $z\in B^2\hs\CC$,
and if yes, gives us an element $c\in C$ such that $z\cdot c\cdot\hm\upgam c=1$.
The details will be given in \cite{BT*}.
Similarly, we can find representatives $z_j^\ssc\in Z^2\hs\CC^\ssc$
of all cohomology classes for $\CC^\ssc$.
This can be done by brute force, because $\CC^\ssc$ is a finite group.
\end{subsec}

\begin{subsec}{\bf Method of calculation.}
\label{ss:m-calc}
For each  $z_j^\ssc\in Z^2\hs\CC$ as above, we compute
$$z_j=\hxx'\cdot \rho_\cC(z_j^\ssc)\in Z^2\hs\CC$$
and determine whether $z_j\in B^2\hs\CC$.
If for all  $z_j^\ssc$ as above we have  $z_j\notin B^2\hs\CC$,
then we conclude that $E$ has no splitting elements.
If $z_j\in B^2\hs\CC$ for some $z_j^\ssc$,
then we find $c\in C$ such that $c\cdot\!\upgam c\cdot z_j=1$.
Let $\TT^\ssc\subset \HH^\ssc$ be a fundamental torus.
By Lemma \ref{l:PR-fundam}(i) we can find
$t^\ssc\in \TT^\ssc(\R)\subset \HH(\R)$ such that $z_j^\ssc=(t^\ssc)^2$.
We set $b=c\hs\rho(t^\ssc)\in H$.
Then
\begin{equation}\label{e:b-upgam-g-hxx'}
\begin{aligned}
b\cdot\!\upgam b\cdot\hxx'=c\cdot\!\upgam c\cdot\rho(t^\ssc)^2\cdot\hxx'
   &=c\cdot\!\upgam c\cdot\rho(z^\ssc_j)\cdot\hxx'\\
   &=c\cdot\!\upgam c\cdot(\hxx')^{-1}\cdot z_j\cdot\hxx'= c\cdot\!\upgam c\cdot z_j=1.
   \end{aligned}
\end{equation}
We set $x''=b x'\in E_\gam$\hs;
then
\[(x'')^2=b\cdot \!\upgam\hs b\cdot (x')^2=b\cdot \!\upgam\hs b\cdot\hxx'=1\]
by \eqref{e:b-upgam-g-hxx'}, and we see that $x''$
is a desired splitting element of $E$.
\end{subsec}

\begin{remark}
This section was extracted from \cite{Borovoi93}.
However, the results of \cite{Borovoi93}
were stated and proved in much greater generality
than in the present article.
For example,  \cite[Proposition 2.3]{Borovoi93},
an analogue of our present Corollary \ref{c:H-ad},
was proved over an arbitrary base field of characteristic 0,
and  \cite[Theorem 5.5]{Borovoi93},
an analogue of our present Theorem \ref{t:spl-rho},
was proved over local fields of characteristic 0
($\R$ and the $p$-adic fields) and over number fields.
For this reason, the corresponding results of \cite{Borovoi93}
were stated and proved as existence theorems,
whereas in the present article
(in Construction \ref{constr:x'} and Subsection \ref{ss:m-calc})
we provide explicit formulas giving a method
of finding a splitting element,
suitable for computer calculations.
\end{remark}

\appendix

\section{Fundamental tori in a simply connected group}
\label{App:PR-fundam}

In this appendix we prove the following lemma:

\begin{lemma}\label{t:fundam}
Let $\GG$ be a simply connected semisimple $\R$-group.
Let $\TT\subset\GG$ be a fundamental torus. Then
\begin{enumerate}
\item[(i)] $\TT$ is a direct product of indecomposable $\R$-tori
isomorphic to $R_{\C/\R}\G_m$ with group of\/ $\R$-points $\C^\times$,
or to $R^{(1)}_{\C/\R}\G_m$  with group of\/ $\R$-points $U(1)$.
\item[(ii)] $H^2\hs\TT=1$.
\end{enumerate}
\end{lemma}

\begin{remark}\label{r:H2}
Any $\R$-torus is a direct product of indecomposable tori isomorphic to
$\G_{m,\R}$ with groups of real points $\R^\times$,
or to $R_{\C/\R}\G_m$\hs, or to $R^{(1)}_{\C/\R}\G_m$;
see Voskresenskii \cite[Section 10.1]{Voskresenskii}.
We have
\[ H^2\hs \G_{m,\R}\cong\{1,-1\}\neq 1,\quad
  H^2\hs R_{\C/\R}\G_m=1,\quad
  H^2\hs R^{(1)}_{\C/\R}\G_m=1.\]
We see that the assertions (i) and (ii) of Lemma \ref{t:fundam} are equivalent.
\end{remark}

\begin{subsec}\label{ss:A-defs}
Let $\GG$ be a {\em compact} real form of a semisimple $\R$-group $G$.
Let $\TT\subset \GG$ be a maximal torus,
and $B\subset G$ be a Borel subgroup containing $T$.
We consider the based root datum
\[\BRD(G,T,B)=(X,X^\vee,R, R^\vee, S, S^\vee);\]
see, for instance, Springer \cite[Sections 1 and 2]{Springer79}. Here
\begin{itemize}
\item[\cc] $X=\X^*(T)$ is the character lattice of $T$,
\item[\cc] $X^\vee=\X_*(T)$ is the cocharacter lattice,
\item[\cc] $R=R(G,T)\subset X$ is the root system,
\item[\cc] $R^\vee\subset X^\vee$ is the coroot system,
\item[\cc] $S=S(G,T,B)\subset R$ is the system of simple roots, and
\item[\cc] $S^\vee\subset R^\vee$ is the system of simple coroots.
\end{itemize}
We write $\BRD(G)$ for $\BRD(G,T,B)$.

There is a canonical homomorphism
\begin{equation}\label{e:psi}
\psi\colon\,\SAut(G)\to\Aut\BRD(G);
\end{equation}
see \cite[Section 3.2]{BKLR}.
Consider the root decomposition
\[\gg=\tl\oplus\bigoplus_{\beta\in R}\gg_\beta\hs,\]
where $\gg=\Lie G$ and $\tl=\Lie T$.
Let $\{X_\alpha\}_{\alpha\in S}$ be a {\em pinning} of $(G,T,B)$,
that is, a family of nonzero elements $X_\alpha\in\gg_\alpha$  for $\alpha \in S$.
Then the restriction of $\psi$ to the subgroup
\[\Aut(G,T,B,\{X_\alpha\})\subset\SAut(G)\]
is an isomorphism; see Conrad \cite[Proposition 1.5.5]{Conrad}.
Inverting this isomorphism, we obtain a splitting of \eqref{e:psi},
that is, a homomorphism
\[\phi\colon\Aut\BRD(G)\isoto \Aut(G,T,B,\{X_\alpha\})\into\SAut(G)\]
such that $\psi\circ\phi=\id_\BRD$.
Moreover, since $\GG$ is compact,
we can choose the pinning $\{X_\alpha\}$ in such a way that
$\phi$ lands in $\Aut_\R(\GG)$; see \cite[Lemma 4.1]{BE16}.

Let $\nu\in\Aut\BRD(G)$ be an automorphism of order 2.
Then $\phi(\nu)\in\Aut_\R(\GG)$ and $\phi(\nu)^2=\id_\GG$.
We see that $\phi(\nu)\in Z^1\hm\Aut(\GG_\C)$, and
we may consider the twisted form $_{\phi(\nu)}\GG$.
\end{subsec}

\begin{definition}
A {\em quasi-compact} real form of a complex semisimple group $G$
is a real group of the form $_{\phi(\nu)}\GG$,
where $\GG$ is a compact form of $G$ and
$\nu\in\Aut\BRD(G)$ is such that $\nu^2=\id$.
\end{definition}

\begin{lemma}\label{l:quasi-compact}
Any real form $\GG'$ of a complex semisimple group $G$
is an inner form of a quasi-compact form of $G$.
\end{lemma}

\begin{proof}
Let $\GG$ be a compact form of $G$.
Then $\GG'=\hs_z\GG$, the twisted form of $\GG$
corresponding to a 1-cocycle $z\in Z^1\hm\Aut_\C(\GG)$,
where we write $\Aut_\C(\GG)$ for $\Aut(G)$
with the $\Gamma$-action defined by the real form $\GG$.
We have a homomorphism
\[\psi\colon\Aut_\C(\GG)\to \Aut\BRD(G),\]
the restriction of the homomorphism \eqref{e:psi}.
This homomorphism is $\Gam$-equivariant.
Note that $\Gamma$ acts trivially on $\Aut\BRD(G)$; see \cite[Lemma 4.2]{BE16}.
We obtain a 1-cocycle $\nu=\psi(z)\in \Aut\BRD(G)$ with $\nu^2=\id$.
Consider the quasi-compact $\R$-group $_{\phi(\nu)}\GG$.
Then $\GG'$ is a twisted form of $_{\phi(\nu)}\GG$.
Since $\psi(z)=\nu=\psi(\phi(\nu))$, we see that $\GG'$
is an {\em inner} form of $_{\phi(\nu)}\GG$, as required.
\end{proof}

\begin{lemma}\label{l:fundamental-inner}
Let $\GG$ be a real semisimple group,
and let $\GG'$ be an {\emm inner} form of $\GG$.
Let $\TT\subset\GG$ and $\TT'\subset \GG'$ be fundamental tori.
Then $\TT\simeq\TT'$.
\end{lemma}

\begin{proof}
A fundamental torus in $\GG$ is the centralizer
of a maximal compact torus $\TT_0\subset\GG$;
see \cite[Section 7]{Borovoi14}.
Since all maximal compact tori in $\GG$ are conjugate,
all fundamental tori in $\GG$ are conjugate as well, and hence isomorphic,
and similarly for $\GG'$.

Let $\TT\subset\GG$ be a fundamental torus,
the centralizer of a maximal compact torus $\TT_0$.
We write $\GG^\ad=\GG/Z(\GG)$ and $\TT^\ad=\TT/Z(\GG)$.
Write $\GG'=\hs_z\GG$, where $z\in Z^1\hs\GG^\ad$.
By \cite[Theorem 1]{Borovoi88},
see also \cite[Theorem 9]{Borovoi14},
we may take $z\in Z^1\hs\TT^\ad$.
Then $_z\TT\subset\hs_z\GG=\GG'$.
Clearly, $_z\TT=\TT$. Thus $\TT$ embeds into $\GG'$.

We show that $\TT$ is a fundamental torus of $\GG'$.
Indeed, $\TT_0\subset\TT\subset\GG'$ is a compact torus in $\GG'$, and therefore,
$\TT_0\subset \TT_0'\subset\GG'$, where $\TT_0'$ is a maximal compact torus of $\GG'$.
Then $\TT_0'$ is contained in the centralizer of $\TT_0$ in $\GG'$, that is, in $\TT$.
We conclude that $\TT$ is a fundamental torus of $\GG'$, as required.
(Since $\TT_0$ is a maximal compact subtorus of $\TT$, we see that $\TT_0'=\TT_0$.)
\end{proof}

\begin{lemma}\label{l:fund-in-quasi-comp}
Let $\GG$ be a {\emm compact, simply connected, semisimple} $\R$-group.
With the notation of Subsection \ref{ss:A-defs},
let $\nu\in\Aut\BRD(G,T,B)$ be an involutive automorphism,
and consider $(\hs_\nu \GG,\hs_\nu\hm\TT)$,
where by abuse of notation we write $\nu$ for $\phi(\nu)$.
Then $_\nu\hm\TT$ is a direct product of indecomposable tori
of the form $R_{\C/\R}\G_m$ and $R^{(1)}_{\C/\R}\G_m$.
\end{lemma}

\begin{proof}
Consider the action of the automorphism $\nu\in\Aut\BRD(G)$ on $S$.
We write
\[S=\{\alpha_1,\dots,\alpha_m,\ \beta_1,\beta_1', \,
     \beta_2,\beta_2',\,\dots,\, \beta_n,\beta_n'\},\]
where $\nu(\alpha_i)=\alpha_i$ and $\nu(\beta_j)=\beta_j'\neq\beta_j$.

Consider the action of $\Gamma$ on the character group
$X^*(T)=\X^*(\TT_\C)$ and on $S\subset \X^*(T)$.
Since the torus $\TT$ is compact, the element $\gamma$
acts by multiplication by $-1$:
\[\upgam\alpha_i=-\alpha_i,\quad \upgam\beta_j=-\beta_j,
        \quad \upgam\beta_j'=-\beta_j'.\]

Consider the twisted action of $\Gamma$ on $\X^*(T)$, that is, the action defined by the
twisted form $_\nu\hm\TT$ of $\TT$:
\[\upgams \alpha_i=\nu(\upgam\alpha_i)=-\alpha_i,\quad
    \upgams\beta_j=\nu(\upgam\beta_j)=-\beta_j',\quad
    \upgams\beta_j'=\nu(\upgam\beta_j')=-\beta_j.\]
It follows that the  similar formulas are true for the simple coroots in $S^\vee$:
\[\upgams \alpha_i^\vee=-\alpha_i^\vee,\quad
    \upgams\beta_j^\vee=-\beta_j^{\hs\prime\hs\vee},\quad
    \upgams(\beta_j^{\hs\prime\hs\vee})=-\beta_j^\vee.\]
Set $\beta_j^{\hs\prime\prime\hs\vee}=-\beta_j^{\hs\prime\hs\vee}$.
Then
\begin{equation}\label{e:action-on-coroots}
\upgams \alpha_i^\vee=-\alpha_i^\vee,\quad
\upgams\beta_j^\vee=\beta_j^{\hs\prime\prime\hs\vee},\quad
\upgams(\beta_j^{\hs\prime\prime\hs\vee})=\beta_j^\vee.
\end{equation}

Since $G$ is simply connected, the set of simple coroots $S^\vee$
is a basis of the cocharacter lattice $\X_*(T)$.
Hence the set
\[\{\alpha_1^\vee,\dots,\alpha_m^\vee,
    \ \beta_1^\vee,\beta_1^{\hs\prime\prime\hs\vee}, \,
    \beta_2^\vee,\beta_2^{\hs\prime\prime\hs\vee},\,\dots,\,
    \beta_n^\vee,\beta_n^{\hs\prime\prime\hs\vee}\},\]
is a basis of  $\X_*(T)$ as well.
Now it follows from \eqref{e:action-on-coroots}
that $_\nu\hm\TT$ is isomorphic to the direct product
of $m$ copies of $R^{(1)}_{\C/\R}\G_m$ and $n$ copies of $R_{\C/\R}\G_m$,
which proves the lemma.
\end{proof}

\begin{lemma} \label{l:fund-in-quasi-comp'}
Under the assumptions of Lemma \ref{l:fund-in-quasi-comp},
the torus $_\nu\hm\TT$ is a fundamental torus of $_\nu\GG$.
\end{lemma}

\begin{proof}
We consider the group  $\nn=\{1,\nu\}$ of order 2 acting on $\TT$.
We prove that $\TT^\nn:=\{t\in T\mid \nu(t)=t\}$ is a maximal compact torus of $_\nu\GG$.
%without using the classification.
Let $\zz$ denote the centralizer of $T^\nn$ in $\gg=\Lie G$.
Then $\zz\supset\tl:=\Lie T$.
It follows that
\[\zz=\tl\oplus\bigoplus_{\eta\in R^\perp}\gg_\eta\hs,\quad
        \text{where}\ \, R^\perp=\{\hs\eta\in R\ \hs\big|\ \hs \eta|_{T^\nn}=1\hs\}.\]
By Lemma \ref{l:perp} below, we have $R^\perp=\emptyset$, whence $\zz=\tl$.
Since the centralizer of a torus in a connected algebraic group
is connected (see Humphreys \cite[Theorem 22.3]{Humphreys}),
we conclude that the centralizer of $T^\nn$ in $G$ coincides with $T$,
and that the centralizer of $\TT^\nn$ in $_\nu\GG$ is $_\nu\hm\TT$.

Let $\SSS$ be any compact torus in $_\nu\GG$ containing $\TT^\nn$.
Then $\SSS$ is contained in the centralizer of $\TT^\nn$
in $_\nu\GG$, and hence  $\SSS\subset \hs_\nu\hm\TT$.
Since $\TT^\nn$ is a maximal compact subtorus of $_\nu\hm\TT$, we conclude that $\SSS=\TT^\nn$.
Thus $\TT^\nn$ is a maximal compact torus of $_\nu\GG$.
Since $_\nu\hm\TT$ contains the maximal compact torus $\TT^\nn$ of $_\nu\GG$,
it is a fundamental torus of $_\nu\GG$, as required.
\end{proof}

\begin{lemma}\label{l:perp}
 $R^\perp=\emptyset$.
\end{lemma}

\begin{proof}
Recall that $X^\vee$ is the cocharacter group of $T$, $X^\vee=\X_*(T)$.
Then $\X_*(T^\nn)=(X^\vee)^\nn$, where by $(\ )^\nn$ we denote the group of fixed points of $\nu$.
Write $V^\vee=X^\vee\otimes_\Z\Q$. Then
\[(V^\vee)^\nn:=\{x\in V^\vee\mid\nu(x)=x\}=(X^\vee)^\nn\otimes_\Z\Q.\]
We have
\[ R^\perp\subset\{\eta\in R\mid \langle\eta,x\rangle=0\ \forall x\in \X_*(T^\nn)\}.\]
Since
\begin{align*}\{\eta\in R\mid \langle\eta,x\rangle=0\ \forall x\in \X_*(T^\nn)\}
      &=\{\eta\in R\mid \langle\eta,x\rangle=0\ \forall x\in (X^\vee)^\nn\}\\
      &=\{\eta\in R\mid \langle\eta,x\rangle=0\ \forall x\in (V^\vee)^\nn\},
\end{align*}
it suffices to show that for any $\eta\in R$ there exists $x_\eta\in (X^\vee)^\nn$
such that $\langle\eta,x_\eta\rangle\neq 0$.

We consider the set $R_+$ of positive roots and the set $R_{\hs-}$ of negative roots in $R$.
Let $\eta\in R_+$. We write
\[\eta=\sum_{i=1}^m a_i\alpha_i +\sum_{j=1}^n (b_j\beta_j+b_j'\beta_j'),\]
where $a_i,b_j,b_j'\ge 0$, and at least one of the coefficients $a_i,b_j,b_j'$ is nonzero.

If $a_i>0$ for some $i$, then we take $x_\eta=\omega_{\alpha,i}^\vee\in V^\vee$,
the fundamental coweight corresponding to $\alpha_i$.
Then
\[\langle\alpha_{i'},\omega_{\alpha,i}^\vee\rangle=\delta_{i,i'}\hs,\quad
     \langle\beta_j,\omega_{\alpha,i}^\vee\rangle=0,\quad \langle\beta_j',\omega_{\alpha,i}^\vee\rangle=0,\]
whence $\langle\eta,x_\eta\rangle=a_i>0$.
Since $x_\eta\in (V^\vee)^\nn$, we conclude that $\eta\notin R^\perp$.

If $b_j>0$ or $b_j'>0$ for some $j$, then $b_j+b_j'>0$.
We take  $x_\eta=\omega_{\beta,j}^\vee+\omega_{\beta,j}^{\hs\prime\hs\vee}\in V^\vee$,
where $\omega_{\beta,j}^\vee$ and $\omega_{\beta,j}^{\hs\prime\hs\vee}$
are the fundamental coweights corresponding to $\beta_j$ and $\beta_j'$, respectively.
Then  $\langle\eta,x_\eta\rangle=b_j+b_j'>0$.
Since  $x_\eta\in (V^\vee)^\nn$,
we conclude that $\eta\notin R^\perp$.

We have proved that $R^\perp\cap R_+=\emptyset$.
We have $-R^\perp=R^\perp$, whence
\[R^\perp\cap R_{\hs-}=(-R^\perp)\cap(-R_+)=-(R^\perp\cap R_+)=\emptyset.\]
Thus $R^\perp=\emptyset$, which completes the proofs of Lemmas \ref{l:perp} and \ref{l:fund-in-quasi-comp'}.
\end{proof}

\begin{proof}[Proof of Lemma \ref{t:fundam}]
By Lemmas \ref{l:fund-in-quasi-comp} and \ref{l:fund-in-quasi-comp'},
the assertion (i) of Lemma \ref{t:fundam} holds in the case when $\GG$ is quasi-compact.
By Lemmas \ref{l:quasi-compact} and \ref{l:fundamental-inner},
this assertion holds for all simply connected semisimple $\R$-groups.
The assertion (ii) of Lemma \ref{t:fundam} follows  from (i); see Remark \ref{r:H2}.
\end{proof}

\textsc{Acknowledgements.}
The author is grateful to Willem de Graaf, Boris Kunyavski\u{\i}, and especially to Andrei Gornitskii,  for very helpful comments.

\end{document}